\documentclass[12pt]{amsart}
\usepackage{color}
\usepackage{amsmath}	
\usepackage{amssymb}

\newcounter{Theorem}

\theoremstyle{plain}
\newtheorem{thm}[Theorem]{Theorem}
\newtheorem{prop}[Theorem]{Proposition}
\newtheorem{cor}[Theorem]{Corollary}
\newtheorem{lemma}[Theorem]{Lemma}

\theoremstyle{remark}

\theoremstyle{definition}
\newtheorem{definition}[Theorem]{\bf Definition}

\numberwithin{Theorem}{section}

\newcommand{\bC}{\mathbb{C}}
\newcommand{\bG}{\mathbb{G}}

\newcommand{\bN}{\mathbb{N}}
\newcommand{\bP}{\mathbb{P}}
\newcommand{\bQ}{\mathbb{Q}}
\newcommand{\bR}{\mathbb{R}}
\newcommand{\bZ}{\mathbb{Z}}
\newcommand{\cA}{{\mathcal{A}}}

\newcommand{\cC}{{\mathcal{C}}}
\newcommand{\cD}{{\mathcal{D}}}

\newcommand{\cL}{{\mathcal{L}}}
\newcommand{\cM}{{\mathcal{M}}}

\newcommand{\cO}{{\mathcal{O}}}
\newcommand{\cS}{{\mathcal{S}}}

\newcommand{\cZ}{{\mathcal{Z}}}

\newcommand{\gp}{\mathfrak{p}}

\newcommand{\talpha}{\tilde{\alpha}}

\newcommand{\tP}{\tilde{P}}

\newcommand{\tZ}{\tilde{Z}}

\newcommand{\ualpha}{{\boldsymbol{\alpha}}}
\newcommand{\ubeta}{{\boldsymbol{\beta}}}

\newcommand{\ugamma}{{\boldsymbol{\gamma}}}

\newcommand{\uomega}{{\boldsymbol{\omega}}}

\newcommand{\utau}{{\boldsymbol{\tau}}}

\newcommand{\uX}{\mathbf{X}}

\newcommand{\uz}{\mathbf{z}}
\newcommand{\uZ}{\underline{Z}}

\newcommand{\disp}{\displaystyle}

\newcommand{\et}{\quad\mbox{and}\quad}
\newcommand{\fleche}{\longrightarrow}

\newcommand{\Ga}{\bG_\mathrm{a}}
\newcommand{\Gm}{\bG_\mathrm{m}}
\newcommand{\habs}{h_{\mathrm{abs}}}

\newcommand{\Qbar}{\overline{\bQ}}
\newcommand{\Res}{\mathrm{Res}}
\newcommand{\tbigwedge}{{\textstyle \bigwedge}}

\DeclareMathOperator{\dist}{dist}

\begin{document}

\baselineskip=14.5pt 

\title[A small value estimate involving translations]
{A small value estimate in dimension two\\ involving translations by rational points}
\author[V. Nguyen]{Ngoc Ai Van Nguyen}
\author[D. Roy]{Damien ROY}
\address{
   D\'epartement de Math\'ematiques\\
   Universit\'e d'Ottawa\\
   585 King Edward\\
   Ottawa, Ontario K1N 6N5, Canada}
\email[Ngoc Ai Van Nguyen]{nnaivan@gmail.com}
\email[Damien Roy]{droy@uottawa.ca}
\subjclass[2010]{Primary 11J85; Secondary 11J81}
\keywords{}
\thanks{Research of both authors partially supported by NSERC}

\begin{abstract}
We show that, if a sequence of non-zero polynomials in $\bZ[X_1,X_2]$
take small values at translates of a fixed point $(\xi,\eta)$ by
multiples of a fixed rational point within the group $\bC\times\bC^*$,
then $\xi$ and $\eta$ are both algebraic over $\bQ$.  The precise
statement involves growth conditions on the degree and norm of
these polynomials as well as on their absolute values at these
translates.  It is essentially best possible in some range of the
parameters.
\end{abstract}


\maketitle

\section{Introduction}
 \label{sec:intro}

Consider the commutative algebraic group $\bC\times\bC^*=
(\Ga\times\Gm)(\bC)$ with group law given by addition on the first
factor and multiplication on the second.  Fix a choice of
\begin{equation}
 \label{intro:eq:points}
 (\xi, \eta)\in \bC\times\bC^* \et (r,s)\in \bQ\times\bQ^*
 \ \text{with}\ r\neq 0 \ \text{and}\ s\neq \pm 1.
\end{equation}
In this paper, we are interested in understanding under which
conditions a non-zero polynomial $P\in\bZ[X_1,X_2]$ can take
small absolute values at translates of $(\xi,\eta)$ by
multiples of $(r,s)$.  We first note that, if $P$ has degree
$D$ and vanishes at the point $(\xi+ir,\eta s^i)$ for
$i=0,1,\dots,D$, then $\xi$ and $\eta$ must be algebraic over
$\bQ$ (see the short argument at the end of Section \ref{sec:height}).
So, it is reasonable that our main result below concludes in
the same way.  Here the norm $\|P\|$ of a polynomial $P$ is the
largest absolute value of its coefficients and, for a real number
$x$, the expression $\lfloor x\rfloor$ denotes its integer part.

\begin{thm}
 \label{intro:thm}
Let $(\xi, \eta)$ and $(r,s)$ be as in \eqref{intro:eq:points},
and let $\sigma, \beta, \nu \in \bR$ satisfying
\begin{equation}
 \label{intro:thm:eq1}
 1 \le \sigma < 2, \quad
 \beta > \sigma +1, \quad
 \nu >
  \begin{cases}
   2+\beta-\sigma &\text{if $\sigma\ge 3/2$,}\\
   \disp 2+\beta-\sigma+\frac{(\sigma-1)(3-2\sigma)}{2+\beta-2\sigma}
    &\text{if $\sigma<3/2$.}
  \end{cases}
\end{equation}
Suppose that, for each sufficiently large positive integer $D$,
there exists a non-zero polynomial $P_D\in \bZ[X_1,X_2]$ such that
\begin{equation}
 \label{intro:thm:eq2}
 \deg P_D \le D, \quad
 \|P_D\| \le e^{D^{\beta}}, \quad
 \max_{0\le i < 4 \lfloor D^{\sigma}\rfloor}
   |P_D(\xi + ir, \eta s^i)| \le e^{-D^{\nu}}.
\end{equation}
Then, $\xi$ and $\eta$ are algebraic over $\bQ$. Moreover,
we have $P_D(\xi + ir, \eta s^i)=0$
$(0 \le i < 4 \lfloor D^{\sigma}\rfloor)$ for each
sufficiently large $D$.
\end{thm}

This statement improves on \cite[Theorem 1.1.5]{Ng} and
is analogous to \cite[Theorem 1.1]{R2012}
in many aspects.  In the latter result, the last of the
conditions \eqref{intro:thm:eq2} is replaced by
\[
 \max_{0\le i < 3 \lfloor D^{\sigma}\rfloor}
   |\cD_1^iP_D(\xi,\eta)| \le e^{-D^{\nu}}
 \quad\text{where}\quad
 \cD_1 = \frac{\partial}{\partial X_1}
          + X_2\frac{\partial}{\partial X_2},
\]
and the constraints on the parameters differ slightly from
\eqref{intro:thm:eq1}.  Although our present result is less
relevant to the conjectures of \cite{R2001}, it has the
following consequence.

\begin{cor}
 \label{intro:cor}
Let $\ell$ be a positive integer, let $(\xi_j, \eta_j)\in\bC\times\bC^*$
for $j=1,\dots,\ell$, and let $(r,s)$, $\sigma$, $\beta$, $\nu$ as in
Theorem \ref{intro:thm}.
Suppose that, for each sufficiently large positive integer $D$,
there exists a non-zero polynomial $P_D\in \bZ[X_1,X_2]$ satisfying
$ \deg P_D \le D$, $\|P_D\| \le e^{D^{\beta}}$ and
\begin{equation}
 \label{intro:cor:eq}
 |P_D(ir+m_1\xi_1+\cdots+m_\ell\xi_\ell,\, s^i\eta_1^{m_1}\cdots\eta_\ell^{m_\ell})|
   \le e^{-D^{\nu}}.
\end{equation}
for any choice of integers $i,m_1,\dots,m_\ell$ with $0\le i\le 4D^\sigma$
and $0\le m_1,\dots,m_\ell\le 2D^{(2-\sigma)/\ell}$.
Then, $1,\xi_1,\dots,\xi_\ell$ are linearly dependent over $\bQ$.
\end{cor}

In \cite[Conjecture 2]{R2001}, it is assumed that $1,\xi_1,\dots,\xi_\ell$
are linearly independent over $\bQ$, and the conclusion is instead an upper
bound on the transcendence degree over $\bQ$ of the field generated
by $\xi_1,\dots,\xi_\ell,\eta_1,\dots,\eta_\ell$.  If, in the present
situation, a similar statement is true, it would require a smaller
value of the parameter $\nu$, for example a value compatible
with Waldschmidt's general construction of auxiliary polynomial
\cite[Th\'eor\`eme 3.1]{Wa} when $(r,s)$, $(\xi_1,\eta_1),\dots,
(\xi_\ell,\eta_\ell)$ all belong to a one-parameter analytic subgroup
of $\bC\times\bC^*$.

To derive Corollary \ref{intro:cor} from Theorem \ref{intro:thm},
we first apply the theorem
to deduce that $\xi_j$ and $\eta_j$ belong to the algebraic closure
$\Qbar$ of $\bQ$ in $\bC$ for $j=1,\dots,\ell$.  By Liouville's
inequality, we then conclude that, for sufficiently large $D$, the
left hand side of \eqref{intro:cor:eq} vanishes for all admissible
choices of $i,m_1,\dots,m_\ell$.  Thus $P_D$ vanishes on the sumset
$\Sigma_D+\Sigma_D$ where $\Sigma_D$ consists of all points
$(ir+m_1\xi_1+\cdots+m_\ell\xi_\ell,\, s^i\eta_1^{m_1}\cdots\eta_\ell^{m_\ell})$
with $0\le i\le 2D^\sigma$ and $0\le m_1,\dots,m_\ell\le D^{(2-\sigma)/\ell}$.
Since the projections of $\Sigma_D$ on both factors of $\bC\times\bC^*$
have cardinality at least $\lfloor 2D^\sigma\rfloor+1 > 2D$, it follows
from Philippon's zero estimate \cite[Th\'eor\`eme 2.1]{PhLZ}, in
the explicit form of \cite[Theorem 5.1]{WaLivre}, that $\Sigma_D$
then has cardinality at most $2D^2$ and so $1,\xi_1,\dots,\xi_\ell$
must be linearly dependent over $\bQ$ (see also \cite[Main Theorem]{MW}).

Coming back to Theorem \ref{intro:thm}, we note that, for $\sigma\ge 3/2$,
our constraint on $\nu$  in \eqref{intro:thm:eq1} is best possible.
Indeed, a simple application of Dirichlet's box principle shows the
following.  If $\sigma,\beta,\nu>0$ satisfy $\sigma<2$, $\beta>2\sigma-1$
and $\nu<2+\beta-\sigma$, then, for each sufficiently large integer
$D$, there exists a non-zero polynomial $P_D\in\bZ[X_1,X_2]$
satisfying \eqref{intro:thm:eq2}, independently of the
nature of $\xi$ and $\eta$.

Moreover, the constraint $\sigma\ge 1$ in \eqref{intro:thm:eq1}
is necessary.  Indeed, if we assume $0<\sigma<1$,
$\beta>2\sigma$ and $\nu>0$, then, a result of
Philippon in \cite[Appendix]{Ph}, adapted from a construction
of Khintchine, shows the existence of two algebraically independent
complex numbers $\xi$ and $\eta$ with the following property.
For each integer $D\ge 1$, there exists a non-zero linear form
$L_D(X_1,X_2)\in\bZ+\bZ X_1+\bZ X_2$ with $\|L_D\|\le D$ and $|L_D(\xi,\eta)|
\le \exp(-D^\nu-D^\beta)$.  Choose an integer $m\ge 1$ such that $mr$ and
$ms^{-1}$ are integers.  Then, for each sufficient large integer
$D$, the polynomial
$P_D=\prod_{0\le i< 4\lfloor D^\sigma\rfloor}(m^iL_D(X_1-ri,X_2s^{-i}))$
has integer coefficients and satisfies \eqref{intro:thm:eq2}.
However, $\xi$ and $\eta$ are transcendental over $\bQ$.

The proof of our main theorem is, for its first part, similar to that
of \cite[Theorem 1.1]{R2012} and we take advantage of this by simply
indicating how the corresponding arguments of \cite{R2012} have to be
modified.  However, there are three new ingredients in the proof which
have independent interest and could be applied to other situations.
The first one, in Section \ref{sec:interp}, is an improvement on
a formula of Mahler from \cite{Mah}.  It implies the interpolation
estimates that we need here as well as the one of \cite{R2012}.
The second ingredient is a simple idea which avoids,
in the present situation, the complicated division process
behind \cite[Proposition 3.7]{R2012}, also used in \cite{Ng}.
It allows us to remove the condition $\nu>2+\sigma$ from
\cite[Theorem 1.1.5]{Ng}.  The third ingredient is a lower bound for
the distance between points from
two distinct $0$-dimensional $\bQ$-subvarieties of $\bP^2(\bC)$
(see Section \ref{sec:height}).  This result allows us to improve the
constraint on $\nu$ from \cite[Theorem 1.1.5]{Ng} and to obtain a
best possible constraint when $\sigma\ge 3/2$ as mentioned above.
However, it does not seem to apply to the situation of \cite{R2012}.
It would be interesting to know if the constraint on $\nu$ in
our main theorem could be improved to $\nu>2+\beta-\sigma$ for
any value of $\sigma$ in the range $1\le \sigma<2$.


\section{Notation and preliminary remarks}
 \label{sec:not}

Recall that we fixed $(\xi, \eta) \in \bC\times\bC^*$
and $(r,s) \in \bQ\times\bQ^*$ with $r\neq 0$ and $s\neq \pm 1$.
However, we will assume from now on that $|s|>1$ because of
the following simple observation.

\begin{lemma}
 \label{not:lemma}
Suppose that Theorem \ref{intro:thm} holds when $|s|>1$.  Then
it holds in general.
\end{lemma}

\begin{proof}
Suppose that the hypotheses of the theorem are satisfied for
some $s$ with $|s|<1$.  We choose $\epsilon>0$ sufficiently small so that the
condition \eqref{intro:thm:eq1} of the theorem
still holds with $\nu$ replaced by $\nu-\epsilon$.  For each sufficiently
large integer $D$, the polynomial $\tP_D = X_2^{\lfloor D/2\rfloor}
P_{\lfloor D/2\rfloor}(X_1,X_2^{-1})$ belongs to $\bZ[X_1,X_2]$,
has degree at most $D$, norm at most $e^{\lfloor D/2\rfloor^\beta}
\le e^{D^\beta}$, and satisfies
\[
 |\tP_D(\xi+ir, \eta^{-1}s^{-i})|
  \le (|\eta|^{-1}|s|^{-i})^{\lfloor D/2\rfloor} e^{-\lfloor D/2\rfloor^\nu}
  \le e^{-D^{\nu-\epsilon}}
  \quad
  (0\le i < 4 \lfloor D^{\sigma}\rfloor).
\]
Since $|s^{-1}|>1$, we conclude
that $\xi$ and $\eta^{-1}$ are algebraic over $\bQ$. So,
$\eta$ is algebraic as well.
\end{proof}

For the proof of the theorem we need to work, as usual,
in a projective setting.  For this reason, we set
\[
 \ugamma _i = (1, \xi + ir, \eta s^i) \in \bC^3
 \et
 \gamma _i = (1: \xi + ir : \eta s^i) \in \bP^2(\bC)
 \quad
 (i\in\bZ)
\]
so that, for each $i$, $\ugamma_i$ is a representative of
$\gamma _i$.  We also denote by $\utau\colon\bC^3\to\bC^3$
the linear map given by
\[
 \utau(x, y, z)=(x, y + r x, sz)
 \quad \text{for each \ } (x,y,z)\in\bC^3,
\]
and by $\tau\colon\bP^2(\bC)\to\bP^2(\bC)$ the corresponding
automorphism of\/ $\bP^2(\bC)$, so that $\utau(\ugamma_i)
= \ugamma_{i+1}$ and $\tau(\gamma_i)=\gamma_{i+1}$ for each
$i\in\bZ$.  Viewing $\bC \times \bC^*$ as a subset of
$\bP^2(\bC)$ under the standard embedding, mapping $(y,z)$
to $(1:y:z)$, the map $\tau$ restricts to translation by
$(r,s)$ in the group $\bC \times \bC^*$.

We also introduce a third variable $X_0$ and set
$\uX=(X_0,X_1,X_2)$ for short.  For each subring $A$ of
$\bC$, we view $A[\uX]$ as an $\bN$-graded ring for the degree
and, for each $D\in\bN$, we denote by $A[\uX]_D$ its homogeneous
part of degree $D$.  We also denote by $\Phi$ the $\bC$-algebra
automorphism of $\bC[\uX]$, homogeneous of degree $0$, given by
\[
 \Phi(P(\uX)) = P(X_0,X_1+rX_0,sX_2) \quad (P \in \bC[\uX]).
\]
In the proof of our theorem, it will play the role of the
derivation $\cD$ of \cite{R2012}.  It satisfies
\[
 \Phi^j(P)(\utau^i(\uz)) = \Phi^{i+j}(P)(\uz)
 \quad (i,j\in\bZ,\, \uz \in \bC^3).
\]
The next result essentially reformulates, in the above
setting, the hypotheses of Theorem \ref{intro:thm}.

\begin{prop}
 \label{not:prop}
Let the notation and hypotheses be as in
Theorem \ref{intro:thm} (with $|s|>1$).  Then, for each
sufficiently large positive integer $D$, there exists
a homogeneous polynomial $\tP_D\in\bZ[\uX]$ of degree $D$,
not divisible by $X_0$ nor by $X_2$, such that
\[
 \Phi^i\tP_D\in\bZ[\uX],\quad
 \|\Phi^i\tP_D\|\le e^{2D^\beta},\quad
 |\tP_D(\ugamma_i)| \le e^{-(1/2)D^\nu}\quad
 (0\le i <4\lfloor D^\sigma\rfloor).
\]
\end{prop}

\begin{proof}
Let $m\in\bN^*$ be a common denominator of $r$ and $s$.
Similarly as in \cite[\S7, Step~1]{R2012}, it suffices to
choose for $\tP_D$ the homogeneous polynomial of $\bZ[\uX]$
of degree $D$ such that $\tP_D(1,X_1,X_2) =
m^{4 D \lfloor D^\sigma\rfloor} X_1^aX_2^{-b} P_D(X_1, X_2)$
where $b$ is the largest integer such that $X_2^b$
divides $P_D$ and $a=D-\deg(P_D)+b$.
\end{proof}

Identifying $\bigwedge^2\bC^3$ with $\bC^3$ as usual and denoting by
$\|\uz\|$ the maximum norm of a point $\uz$ in $\bC^3$, we define
the {\it distance} between two points $\alpha, \beta \in
\bP^2(\bC)$ by
\[
 \dist(\alpha, \beta)
 = \frac{\|\ualpha \wedge \ubeta\|}{\|\ualpha\| \,  \|\ubeta\|},
\]
where $\ualpha$ and $\ubeta$ are representatives of $\alpha$ and $\beta$
in $\bC^3$.  This is not properly a distance function on $\bP^2(\bC)$
but it satisfies the following weak triangle inequality
\[
 \dist(\alpha,\gamma)\le 2\dist(\alpha,\beta)+\dist(\beta,\gamma)
 \quad
 (\alpha,\beta,\gamma\in\bP^2(\bC)).
\]
More generally, for any $\alpha\in\bP^2(\bC)$ and any finite
subset $\cS$ of\/ $\bP^2(\bC)$, we define the distance from $\alpha$
to $\cS$, denoted $\dist(\alpha,\cS)$, as the smallest distance
between $\alpha$ and a point of $\cS$.  We conclude this section with
two simple facts.

\begin{lemma}
 \label{not:lemma2}
There exists a constant $c_1=c_1(r,s) > 0$ such that
\[
 |\log\,\dist(\tau^j(\alpha),\tau^j(\beta))
   - \log\,\dist(\alpha,\beta) |
 \le c_1 |j|
\]
for any distinct points $\alpha,\beta\in\bP^2(\bC)$
and any $j\in\bZ$.
\end{lemma}

\begin{proof}
Since $\utau$ and $\bigwedge^2\utau$ are invertible linear maps,
there exists a constant $c=c(r,s)> 1$ such that
\[
 c^{-1}\|\ualpha\| \le \|\utau(\ualpha)\| \le c\|\ualpha\|
 \et
 c^{-1}\|\uomega\| \le \big\|\tbigwedge^2\utau(\uomega)\big\|
 \le c\|\uomega\|
\]
for any $\ualpha\in\bC^3$ and $\uomega\in\bigwedge^2\bC^3$.  Thus,
for any distinct $\alpha,\beta\in\bP^2(\bC)$ with representatives
$\ualpha,\ubeta\in\bC^3$, we find
\[
 \frac{\dist(\alpha,\beta)}{\dist(\tau(\alpha),\tau(\beta))}
 = \frac{\|\utau(\ualpha)\|}{\|\ualpha\|}
   \frac{\|\utau(\ubeta)\|}{\|\ubeta\|}
   \frac{\|\ualpha\wedge\ubeta)\|}
        {\|\tbigwedge^2 \utau(\ualpha\wedge\ubeta)\|}
 \in [c^{-3},\, c^3].
\]
Therefore, the estimate of the lemma holds for $j=1$ with
$c_1=3\log(c)$.  By induction on $j$, it holds for any $j\ge 0$.
Thus it also holds for $j<0$ (upon replacing $\alpha$ and
$\beta$ by their images under $\tau^{-j}$).
\end{proof}

\begin{lemma}
 \label{not:lemma3}
Let $\alpha,\beta\in\bP^2(\bC)$ with representatives
$\ualpha,\ubeta\in\bC^3$ of norm $1$, let $D\in\bN$ and let
$P \in \bC[\uX]_D$.  Then, we have
$|P(\ualpha)| \le |P(\ubeta)| + D \cL(P)\,\dist(\alpha, \beta)$.
\end{lemma}

Here, $\cL(P)$ stands for the \emph{length} of $P$, namely
the sum of the absolute values of the coefficients of $P$.

\begin{proof}
It is well known that
\[
 |P(\ualpha)Q(\ubeta) - P(\ubeta)Q(\ualpha)|
  \le  D \cL(P)\cL(Q)\,\dist(\alpha, \beta)
\]
for any $Q\in\bC[\uX]_D$.  We apply this to $Q=X_i^D$ with
$i\in\{0,1,2\}$ chosen so that $|Q(\ubeta)|=1$.  Then, the
result follows because $\cL(Q)=1$ and $|Q(\ualpha)|\le 1$.
\end{proof}


\section{Interpolation estimate} \label{sec:interp}

We start with the following improvement of Malher's useful
formula (7) from page 88 of \cite{Mah}, also stated
as Lemma 2 in \cite{Ti}, which we phrase in terms
of linear recurrence sequences using the notation
\[
 i^{(0)}=1,\qquad
 i^{(\mu)} = i(i-1) \cdots (i-\mu+1)
  \quad (i\in\bN,\mu\in\bN^*),
\]
and the convention that the empty product is $1$, in particular
$z^0=1$ for any $z\in\bC$.

\begin{lemma}
  \label{interp:lemma1}
Consider the linear recurrence sequence $u=(u_i)_{i\in\bN}$
given by
\[
 u_i = \sum_{\nu=0}^{n-1}\sum_{\mu=0}^{m_{\nu}-1}
          A_{\mu,\nu} i^{(\mu)} \alpha_\nu^{i-\mu}
 \quad (i\in\bN),
\]
for fixed $n\in\bN^*$, $\alpha_\nu\in\bC$, $m_\nu\in\bN^*$
and $A_{\mu,\nu}\in\bC$, with $\alpha_0,\dots,\alpha_{n-1}$
distinct.  Set
\[
 \begin{aligned}
 M&= \sum_{\nu=0}^{n-1} m_\nu,
 &a_0 &= \left(\max_{0\le \nu<n} \binom{M}{m_\nu}\right)
        \prod_{\nu =0}^{n-1} (1+ |\alpha_\nu|)^{m_\nu}, \\
 a_1 &= \min_{0 \le \nu <n}
       \prod_{\substack{0\le \nu'<n \\ \nu' \neq \nu}}
           |\alpha_{\nu'} - \alpha_{\nu}|^{m_{\nu'}},
 &a_2 &= \min_{\substack{0\le \nu,\nu'<n \\ \nu \neq \nu'}}
           \left
           \{1,|\alpha_{\nu'} - \alpha_{\nu}|^{m_{\nu}}
           \right\},
 \end{aligned}
\]
with the understanding that $a_1=a_2=1$ if $n=1$.
Then, we have
\[
 \max_{\substack{0\le\nu<n \\ 0\le\mu<m_\nu}}|A_{\mu,\nu}|
   \le \frac{a_0}{a_1 a_2} \max_{0\le i < M}|u_i|.
\]
\end{lemma}

The  connection with Mahler's formula is that $u_i=E^{(i)}(0)$
where $E$ is the exponential polynomial given by
$E(z)=\sum_{\mu,\nu}A_{\mu,\nu}z^\mu e^{\alpha_{\nu} z}$
($z\in\bC$).  The above result improves on Mahler's
because $a_0\le (4a)^M$ where $a=\max_\nu\{1,|\alpha_\nu|\}$,
while the inequality of Malher has $a_0$ replaced by $2(6a)^M$.
As we will see below, this improvement allows us to gain an
order of magnitude for the application that we have in view.
The proof of the lemma follows that of \cite[Prop.~3.3]{R2012}.

\begin{proof}
We have $u=\sum_{\mu,\nu}A_{\mu,\nu} u^{(\mu,\nu)}$ with sequences
$u^{(\mu,\nu)}=(u^{(\mu,\nu)}_i)_{i\in\bN}$ given by
\[
 u_i^{(\mu,\nu)}= i^{(\mu)} \alpha_{\nu}^{i-\mu}
 \quad (i \in \bN,\, 0\le \nu<n,\, 0\le \mu< m_\nu).
\]
Let $\tau $ denote the linear operator on $\bC^{\bN}$ which
sends a sequence $(x_i)_{i \in \bN}$
to the shifted sequence $(x_{i+1})_{i \in \bN}$.  We first
note that, for each $(\mu,\nu)$ with $0\le \nu<n$ and
$0\le \mu<m_\nu$, we have
\begin{equation*}
 \label{interp:lemma1:eq1}
 (\tau-\alpha_{\nu})^m  u^{(\mu,\nu)} = \mu^{(m)} u^{(\mu-m,\nu)}
 \quad (0\le m\le \mu),
\end{equation*}
with initial term
$\left((\tau-\alpha_{\nu})^m  u^{(\mu,\nu)}\right)_0
  = \mu! \delta_{m,\mu}$.
Moreover, $(\tau-\alpha_{\nu})^{m_\nu}$ annihilates
$u^{(\mu,\nu)}$.  Now, fix a choice of $(\mu,\nu)$ as above,
and form
\[
 c(X)= \frac{1}{\mu!}(X-\alpha_{\nu})^{\mu}
       \prod_{\substack{0\le\nu'<n \\ \nu' \neq \nu}}
          \left(
            \frac{X-\alpha_{\nu'}}{\alpha_{\nu} -\alpha_{\nu'}}
          \right)^{m_{\nu'}}.
\]
By \cite[Lemma 3.2]{R2012}, there exists a unique polynomial
$a(X)\in\bC[X]$ of degree at most $m_\nu -\mu -1$ such that
\[
 a(X) \prod_{\substack{0\le\nu'<n \\ \nu' \neq \nu}}
      \left( 1-\frac{X}{\alpha_{\nu'} -\alpha_{\nu}} \right)^{m_{\nu'}}
 \equiv 1
 \mod X^{m_\nu-\mu}
\]
and its length satisfies
\[
 \cL(a)
 \le \binom{M-\mu-1}{m_\nu-\mu-1}
     \max_{\substack{0\le\nu'<n \\ \nu' \neq \nu}}
        \left\{ 1, \, \frac{1}{|\alpha_{\nu'} - \alpha_{\nu}|}
        \right\}^{m_\nu-\mu-1}
 \le \binom{M}{m_\nu} \frac{1}{a_2}.
\]
Then $b(X)=a(X-\alpha_\nu)c(X)$ is a polynomial of degree
at most $M-1$ which is divisible by $(X-\alpha_{\nu'})^{m_{\nu'}}$
for each $\nu'=0,\dots,n-1$ with $\nu'\neq \nu$, and which is
congruent to $(\mu!)^{-1}(X-\alpha_\nu)^\mu$ modulo
$(X-\alpha_{\nu})^{m_{\nu}}$.  In view of the above, this implies
that $A_{\mu,\nu}=(b(\tau)u)_0$ and so
\[
 |A_{\mu,\nu}|
   \le \cL(b)\max_{0\le i<M} |u_i|.
\]
The conclusion follows because
$\cL(b)\le \cL(a)(1+|\alpha_\nu|)^{m_\nu-\mu-1}\cL(c)
\le a_0/(a_1a_2)$, and because the choice of $(\mu,\nu)$ is
arbitrary.
\end{proof}

With the help of the above result, one easily recovers
Proposition 3.3 of \cite{R2012} up to the values of the
constants.  In our context, it has the following consequence.

\begin{prop}
 \label{interp:prop1}
Let $L \in \bN$ and $M = \binom{L+2}{2}$.  Then the map
\begin{align*}
  \bC[\uX]_L &\longrightarrow \bC^M \\
  Q \quad & \longmapsto (Q(\ugamma_i))_{0\le i<M}
\end{align*}
is an isomorphism of $\bC$-vector spaces.  Moreover, there
exists a constant $c_2 \ge 3$ depending only on $r,s,\xi,\eta$
such that any $Q \in \bC[\uX]_L$ satisfies
\begin{equation}
 \label{interp:prop1:eq1}
  \cL(Q) \le (c_2L)^{3L} \max_{0\le i<M} |Q(\ugamma_i)| .
\end{equation}
\end{prop}

\begin{proof}
We simply prove the second assertion as it implies the
first.  To this end, we fix a polynomial $Q \in \bC[\uX]_L$.
Then there exists $P \in \bC[\uX]_L$ such that
\[
 Q(\uX) = P(X_0,r^{-1}(X_1-\xi X_0),\eta^{-1}X_2).
\]
Writing $P$ in the form
\[
 P(\uX)
  = \sum_{\nu=0}^{L}\sum_{\mu=0}^{L-\nu}
      p_{\mu,\nu} X_0^{L-\mu-\nu}
        X_1(X_1-X_0) \cdots (X_1-(\mu-1)X_0)X_2^{\nu}
\]
with $p_{\mu,\nu} \in \bC$, we find
\[
 \cL(Q) \le c^L\cL(P) \le c^L M L!\,\max|p_{\mu,\nu}|
\]
for some constant $c=c(r,\xi,\eta)>0$, and
\[
 Q(\ugamma_i)
   = P(1,i,s^i)
   =  \sum_{\nu=0}^{L}\sum_{\mu=0}^{L-\nu}
        A_{\mu,\nu} i^{(\mu)} (s^{\nu})^{i-\mu}
   \quad (0\le i),
\]
where $A_{\mu,\nu}=p_{\mu,\nu} s^{\mu\nu}$.  By Lemma
\ref{interp:lemma1}, this gives
\[
 \max |p_{\mu,\nu}|
   \le \max |A_{\mu,\nu}|
   \le \frac{a_0}{a_1 a_2} \max_{0\le i < M}|Q(\ugamma_i)|
\]
where
\begin{align*}
 a_2 &= \min_{\nu' \neq \nu} \{1,\,|s^{\nu'} - s^{\nu}|\}^{L-\nu+1}
      \ge \min\{1,|s|-1\}^{L+1}\\
\noalign{\noindent\text{and}}
 \frac{a_0}{a_1}
   &= \binom{M}{L+1} \max_{0\le \nu\le L}
      (1+|s^{\nu}|)^{L-\nu+1}
      \prod_{\substack{0\le \nu'\le L \\ \nu' \neq \nu}}
        \left(
          \frac{1+|s^{\nu'}|}{|s^{\nu'}-s^{\nu}|}
        \right)^{L-\nu'+1}\\
  &\le (M+1)^{L+1}
      (c')^{L+1} \max_{0\le \nu\le L} |s|^{\nu(L-\nu+1)}
      \prod_{0\le \nu'<\nu} |s|^{(\nu'-\nu)(L-\nu'+1)}\\
  &=(M+1)^{L+1} (c')^{L+1},
\end{align*}
with $\disp c'=\prod_{i=1}^\infty \frac{1+|s|^{-i+1}}{(1-|s|^{-i})^2}$.
The conclusion follows.
\end{proof}

We don't know if the factor $(c_2L)^{3L}$ in \eqref{interp:prop1:eq1}
has optimal order.  However, the formula of Mahler which we
discussed after Lemma \ref{interp:lemma1} yields
instead a multiplier of the order of $c^{L^3}$.  This would
have been sufficient for our purpose, but we think that the
above estimate, whose proof does not require much more work,
is interesting in itself.  In the case where $|s|<1$, the first
author shows in \cite[Lemma 1.5.1]{Ng} that \eqref{interp:prop1:eq1}
holds with $(c_2L)^{3L}$ replaced by $c^{L^3}$ for a constant
$c>1$ and that this is optimal up to the value of that constant
\cite[Example 1.5.2]{Ng}.

\begin{definition}
For each $T\in\bN^*$, we denote by $I^{(T)}$ the homogeneous
ideal of $\bC[\uX]$ generated by the homogeneous
polynomials vanishing on $\{\ugamma_i\,;\,0\le i<T\}$.
For each $D\in\bN$, we denote by $I^{(T)}_D$ the homogeneous
part of $I^{(T)}$ of degree $D$ and, for each $\alpha \in
\bP^2(\bC)$ with representative $\ualpha\in\bC^3$ of maximum
norm $\|\ualpha\|=1$, we define
\[
 |I^{(T)}_D|_{\alpha}
   = \sup \{|P(\ualpha)| \,;\, \ P \in I^{(T)}_D,\, \|P \| \le 1 \},
\]
the right end side being independent of the choice of $\ualpha$.
\end{definition}

The next result is the counterpart of \cite[Prop.~4.5]{R2012}.
Its proof however is much simpler and uses a genuinely different
argument.  It allows us to avoid the condition $\nu>2+\sigma$
which is needed in \cite{Ng}.

\begin{prop}
 \label{interp:prop:dist}
Let $D,T\in\bN^*$ with $T \le \binom{D+1}{2}$, and let
$\alpha\in\bP^2(\bC)$.  Then we have
\[
 \dist\big(\alpha,\{\gamma_0,\dots,\gamma_{T-1}\}\big)
  \le c_3^{T^{3/2}}|I_D^{(T)}|_\alpha
\]
where $c_3\ge 3$ depends only on $r$, $s$, $\xi$, $\eta$.
\end{prop}

\begin{proof}
Let $\ualpha=(\alpha_0,\alpha_1,\alpha_2)\in\bC^3$ be a
representative of $\alpha$ of norm $1$ and let $k\in\{0,1,2\}$
such that $|\alpha_k|=1$.  We denote by $L$ the smallest non-negative
integer with $T\le \binom{L+2}{2}$ and set $M=\binom{L+2}{2}$.
Then, by hypothesis, we have $0\le L<D$.  According to Proposition
\ref{interp:prop1}, there exists a basis $(Q_0,\dots,Q_{M-1})$
of $\bC[\uX]_L$ such that $Q_j(\ugamma_i)=\delta_{i,j}$ for any pair
of indices $(i,j)$ with $0\le i,j<M$.  Moreover, these polynomials
have length $\cL(Q_j)\le (c_2L)^{3L}$ for $j=0,\dots,M-1$.  Write
\begin{equation}
 \label{interp:prop:dist:eq1}
 X_k^L = \sum_{j=0}^{M-1}a_jQ_j(\uX)
\end{equation}
with $a_0,\dots,a_{M-1}\in\bC$, and let $i$ denote the index for
which $|a_iQ_i(\ualpha)|$ is maximal.  By construction, $a_i$ is
the value of the polynomial $X_k^L$ at the point $\ugamma_i$, and
so we have $|a_i|\le \|\ugamma_i\|^L$.  Then, upon evaluating both
sides of \eqref{interp:prop:dist:eq1} at the point $\ualpha$,
we obtain
\begin{equation}
 \label{interp:prop:dist:eq2}
 1=|\alpha_k|^L
   \le M |a_i|\, |Q_i(\ualpha)|
   \le M \|\ugamma_i\|^L |Q_i(\ualpha)|.
\end{equation}
Suppose first that $i<T$.  Then, we denote by $E(\uX)$ one of
the linear forms
\[
 X_1-(\xi+ir)X_0,\quad
 X_2-\eta s^iX_0,\quad
 (\xi+ir)X_2-\eta s^iX_1
\]
for which $|E(\ualpha)|=\|\ugamma_i\|\dist(\alpha,\gamma_i)$,
and we set $P(\uX)=X_k^{D-L-1}E(\uX)Q_i(\uX)$.  Since $L<D$
and $E(\ugamma_i)=0$, we have $P(\uX)\in I^{(M)}_D
\subseteq I^{(T)}_D$, and so, by definition,
\begin{equation}
 \label{interp:prop:dist:eq3}
 |P(\ualpha)|\le \cL(P)|I^{(T)}_D|_\alpha.
\end{equation}
As $|P(\ualpha)|=\|\ugamma_i\|\dist(\alpha,\gamma_i)|Q_i(\ualpha)|$
and $\cL(P)\le \cL(E)\cL(Q_i) \le 2\|\ugamma_i\|(c_2L)^{3L}$, this
together with \eqref{interp:prop:dist:eq2} yields
\[
 \dist(\alpha,\gamma_i)
   \le \frac{2 (c_2L)^{3L}}{|Q_i(\ualpha)|}\,|I^{(T)}_D|_\alpha
   \le 2M \|\ugamma_i\|^L (c_2L)^{3L} |I^{(T)}_D|_\alpha.
\]
In the complementary case where $i\ge T$, we set
$P(\uX)=X_k^{D-L}Q_i(\uX)$.   Then $P$ belongs to $I^{(T)}_D$ since
$Q_i\in I^{(T)}_L$, and so \eqref{interp:prop:dist:eq3} holds again.
Using \eqref{interp:prop:dist:eq2}, this gives
\[
 1
   \le \frac{(c_2L)^{3L}}{|Q_i(\ualpha)|}\,|I^{(T)}_D|_\alpha
   \le M \|\ugamma_i\|^L (c_2L)^{3L} |I^{(T)}_D|_\alpha.
\]
So, in both cases, we obtain
\[
 \dist\big(\alpha,\{\gamma_0,\dots,\gamma_{T-1}\}\big)
  \le 2M \|\ugamma_i\|^L (c_2L)^{3L} |I^{(T)}_D|_\alpha
\]
using the fact that the distance in the left hand side is
bounded above by $2$.
The conclusion follows because $i<M\le 2T$, $L\le \sqrt{2T}$
and $\|\ugamma_i\|\le c|s|^i$ for a constant $c>0$.
\end{proof}


\section{Preliminaries on heights}
 \label{sec:height}

Our notation differs slightly from \cite{LR} and \cite{R2012}.
For any set $S$ of homogeneous polynomials of $\bC[\uX]$,
we denote by $\cZ(S)$ the set of common zeros of the elements
of $S$ in $\bP^2(\bC)$.  We define a \emph{$\bQ$-subvariety}
of\/ $\bP^2(\bC)$ to be a non-empty subset of the form $\cZ(\gp)$
for some homogeneous prime ideal $\gp$ of $\bQ[\uX]$.

Let $Z$ be a $\bQ$-subvariety of\/ $\bP^2(\bC)$, let $t=\dim(Z)$ denote its
dimension, and let $D$ be a positive integer.  The \emph{Chow form}
of $Z$ in degree $D$ is the polynomial map
$
 F \colon \bC[\uX]_D^{t+1} \fleche \bC
$
characterized uniquely up to multiplication by $\pm 1$ by the
following two properties:
\begin{itemize}
 \item[1)] its zeros are the $(t+1)$-tuples of
  polynomials $(P_0,\dots,P_t) \in \bC[\uX]_D^{t+1}$
  having at least one common zero on $Z$,
 \item[2)] its underlying polynomial relative to the basis of
  $(t+1)$-tuples of pure monomials $\uX^{\nu}$
  of degree $D$ has coefficients in $\bZ$ and is irreducible over $\bZ$.
\end{itemize}
We define the \emph{height} $h(Z)$ of $Z$ as the logarithm of
the norm of the latter polynomial when $D=1$.  We also denote by
$\deg(Z)$ the \emph{degree} of $Z$.  This is the cardinality of the
intersection of $Z$ with a generic linear subvariety of \/$\bP^2(\bC)$
of codimension $t$.  It is also characterized by the fact that
the Chow form of $Z$ in degree $1$ is separately homogeneous of
degree $\deg(Z)$ in each of its $t+1$ polynomial arguments.

For example, suppose that $Z$ has dimension $0$.  Let
$\ualpha=(\alpha_0,\alpha_1,\alpha_2)\in\Qbar^3$ be a representative
of a point of $Z$ chosen so that at least one of its coordinates
$\alpha_k$ is equal to $1$.  Then $\deg(Z)$ is the degree $n=[K:\bQ]$
of the field $K=\bQ(\ualpha)$ over $\bQ$.  Moreover, if $\sigma_i\colon
K \to \bC$ ($1\le i\le n$) are all the embeddings of $K$ into
$\bC$, then
\[
 \sigma_i(\ualpha)
 :=(\sigma_i(\alpha_0),\sigma_i(\alpha_1),\sigma_i(\alpha_2))
 \quad (1\le i\le n)
\]
are representatives
of the $n$ points of $Z$.  Moreover, the Chow form of $Z$ in
degree 1 is the map $F\colon\bC[\uX]_1\to\bC$ given by
\[
 F(L) = a\prod_{i=1}^n L(\sigma_i(\ualpha))
\]
for an appropriate non-zero integer $a$ chosen so that the polynomial
underlying $F$ has content $1$.  Then, it is easy to compare the height
of $Z$ with the absolute logarithmic Weil height of $\ualpha$
defined as
\[
 \habs(\ualpha)
 = \frac{1}{n}\sum_{\nu\in\cM(K)} [K_\nu:\bQ_\nu] \log\|\ualpha\|_\nu
\]
where $\nu$ runs through the set $\cM(K)$ of all places of $K$,
where, for each $\nu\in\cM(K)$, the fields $K_\nu$ and $\bQ_\nu$
are respectively the completions of $K$ at $\nu$ and of $\bQ$ at the
place of $\bQ$ induced by $\nu$, and where $\|\ \|_\nu$ stands for
the maximum norm on $K_\nu^3$.  We first note that, by the choice
of $a$, we have
\[
 \habs(\ualpha)
 = \frac{1}{n}\log|a| + \frac{1}{n}\sum_{i=1}^n \log\|\sigma_i(\ualpha)\|
\]
and so, Gel'fond's inequality \cite[Ch.~III, \S3, Lemma II]{Ge}
relating the norm of a product of
polynomials to the product of their norms yields
\[
 |n^{-1}h(Z)-\habs(\ualpha)| \le 3.
\]
We will derive several consequences of this
estimate, starting with the following result.

\begin{lemma}
 \label{height:lemma1}
Let $Z$ be a $\bQ$-subvariety of\/ $\bP^2(\bC)$ of
dimension $0$.  For each $i\in\bZ$, the translate $\tau^i(Z)$
is a $\bQ$-subvariety of \/$\bP^2(\bC)$ with
\[
 \deg(\tau^i(Z))=\deg(Z)
 \et
 \big|h(\tau^i(Z))-h(Z)\big| \le c_4|i|\deg(Z)
\]
for some $c_4=c_4(r,s)>0$.
\end{lemma}

\begin{proof}
It suffices to prove this for $i=1$.  Let $\gp$
be the prime ideal of $\bQ[\uX]$ defining $Z$.  As $\Phi$
restricts to an automorphism of $\bQ[\uX]$, the set
$\Phi^{-1}(\gp)$ is also a prime ideal of $\bQ[\uX]$.
Thus $\tau(Z)=\cZ(\Phi^{-1}(\gp))$ is a $\bQ$-subvariety
of \/$\bP^2(\bC)$.  It has the same degree $n$ as $Z$ because
the map $\tau$ is a bijection from $\bP^2(\bC)$ to itself.

Let $\ualpha=(\alpha_0,\alpha_1,\alpha_2)$ be a representative
of a point of $Z$ with at least one coordinate equal to $1$, and
let $K=\bQ(\ualpha)$.  Then
$\utau(\ualpha)=(\alpha_0,r\alpha_0+\alpha_1,s\alpha_2)$
is a representative in $K^3$ of a point of $\tau(Z)$, and so
\[
 \frac{1}{n} |h(\tau(Z))-h(Z)|
   \le 6 + |\habs(\tau(\ualpha))-\habs(\ualpha)|.
\]
On the other hand, for each place $\nu$ of $K$, we have
\[
 \|\tau(\ualpha)\|_\nu
  \le 2^{\epsilon_\nu} \max\{1,|r|_\nu\} \max\{1,|s|_\nu\} \|\ualpha\|_\nu.
\]
where $\epsilon_\nu=1$ if $\nu|\infty$ and $\epsilon_\nu=0$ otherwise.
This yields
\[
 \habs(\utau(\ualpha)) \le c + \habs(\ualpha)
\]
where $c=\log(2)+\habs(1,r)+\habs(1,s)$.  Similarly, we find that
$\habs(\ualpha)\le c + \habs(\utau(\ualpha))$, and so we conclude that
$|h(\tau(Z))-h(Z)| \le c_4n$ where $c_4=6+c$.
\end{proof}

\begin{prop}
 \label{height:prop}
Let $Z,Z^*$ be distinct $\bQ$-subvarieties of \/$\bP^2(\bC)$ of dimension $0$,
and let $\cA$ be any subset of $Z\times Z^*$.  Then,
\[
 \sum_{(\alpha,\alpha^*)\in\cA} \log \dist(\alpha,\alpha^*)
 \ge -7\deg(Z)\deg(Z^*)-\deg(Z)h(Z^*)-\deg(Z^*)h(Z).
\]
\end{prop}

\begin{proof}
The set $Z\times Z^*$ is invariant under the Galois
group of $\Qbar$ over $\bQ$.  Let $\cO$ be one of its orbits and
let $(\gamma,\gamma^*)\in\cO$.  Since $Z$ and $Z^*$ are distinct,
the points $\gamma$ and $\gamma^*$ are also distinct.
Choose representatives $\ugamma$ of $\gamma$ and $\ugamma^*$ of
$\gamma^*$ having at least one coordinate equal to $1$ and set
$K=\bQ(\ugamma,\ugamma^*)$.  Then the cardinality of $\cO$ is
$|\cO|=[K:\bQ]$.  Moreover, for each place $\nu\in\cM(K)$,
we have
\[
 \frac{\|\ugamma\wedge\ugamma^*\|_\nu}%
      {\|\ugamma\|_\nu\|\ugamma^*\|_\nu}
 \le \begin{cases}
       1 &\text{if\, $\nu\nmid\infty$,}\\
       2 &\text{if\, $\nu\mid\infty$.}
     \end{cases}
\]
From this we deduce that
\begin{align*}
 \sum_{\nu\in\cM(K)}
   \frac{[K_\nu:\bQ_\nu]}{[K:\bQ]}
   \log \frac{\|\ugamma\wedge\ugamma^*\|_\nu}%
         {\|\ugamma\|_\nu\|\ugamma^*\|_\nu}
 &\le \log(2)
   + \sum_{\stackrel{\scriptstyle\sigma\colon K\hookrightarrow \bC}%
          {(\sigma(\gamma),\sigma(\gamma^*))\in\cA}}
         \frac{1}{[K:\bQ]}
         \log \frac{\|\sigma(\ugamma)\wedge\sigma(\ugamma^*)\|}%
               {\|\sigma(\ugamma)\|\,\|\sigma(\ugamma^*)\|}\\
 &= \log(2) +
   \frac{1}{|\cO|} \sum_{(\alpha,\alpha^*)\in\cA\cap\cO}
   \log \dist(\alpha,\alpha^*).
\end{align*}
On the other hand, we have
\[
 \sum_{\nu\in\cM(K)}
   \frac{[K_\nu:\bQ_\nu]}{[K:\bQ]}
   \log \frac{\|\ugamma\wedge\ugamma^*\|_\nu}%
         {\|\ugamma\|_\nu\|\ugamma^*\|_\nu}
 =\habs(\ugamma\wedge\ugamma^*)-\habs(\ugamma)-\habs(\ugamma^*).
\]
Since $\habs(\ugamma\wedge\ugamma^*)\ge 0$, the combination of these two
estimates yields
\begin{align*}
 \frac{1}{|\cO|} \sum_{(\alpha,\alpha^*)\in\cA\cap\cO}
   \log \dist(\alpha,\alpha^*)
 &\ge -\habs(\ugamma)-\habs(\ugamma^*) - \log(2)\\
 &\ge -\frac{h(Z)}{\deg(Z)}-\frac{h(Z^*)}{\deg(Z^*)} - 7.
\end{align*}
Summing over all orbits $\cO$, this gives
\begin{align*}
 \sum_{(\alpha,\alpha^*)\in\cA}
   \log \dist(\alpha,\alpha^*)
 &\ge -\left(\sum |\cO|\right)
       \left(\frac{h(Z)}{\deg(Z)}+\frac{h(Z^*)}{\deg(Z^*)}+7\right)\\
 &= -7\deg(Z)\deg(Z^*)-\deg(Z)h(Z^*)-\deg(Z^*)h(Z),
\end{align*}
since $\sum |\cO| = |Z\times Z^*| =\deg(Z)\deg(Z^*)$.
\end{proof}

We conclude this section with the following counterpart
to \cite[Lemma 5.4]{R2012}, which readily implies the
assertion made in the introduction just before the statement
of Theorem \ref{intro:thm}.

\begin{lemma}
 \label{height:lemma:gcd}
Let $D\in\bN^*$ and $P\in\bC[\uX]_D$.  Suppose that $P$
is not divisible by $X_0$ nor by $X_2$.  Then the polynomials
$P,\Phi(P),\dots,\Phi^D(P)$ have no common irreducible factor
in $\bC[\uX]$.
\end{lemma}

\begin{proof}
Suppose on the contrary that these polynomials have a common
irreducible factor $Q$.  Then $Q$ is homogeneous of some
degree $t\ge 1$.  Since $\Phi$ is a degree preserving
$\bC$-algebra automorphism of $\bC[\uX]$, we deduce that
$\Phi^{-i}(Q)$ is an homogeneous irreducible factor of $P$
of degree $t$ for $i=0,\dots,D$.  Since $\deg(P)<(D+1)t$,
two of these factors must be associates.  Thus there exists
$k\in\bZ$ with $k\neq 0$ such that $\Phi^k(Q)=\lambda Q$
for some $\lambda\in \bC^*$. In other words, $Q$ is an
eigenvector for the restriction of $\Phi^k$ to $\bC[\uX]_t$.
However, $\bC[\uX]_t=\oplus_{i=0}^t X_2^i\bC[X_0,X_1]_{t-i}$
is a direct sum decomposition of $\bC[\uX]_t$ into invariant
subspaces for $\Phi^k$ and, for each $i=0,\dots,t$, the
restriction of $\Phi^k$ to $X_2^i\bC[X_0,X_1]_{t-i}$ admits
$s^{ik}$ as its only eigenvalue, with $\bC X_0^{t-i}X_2^i$
as its corresponding eigenspace.  Since the numbers
$1,s^k,\dots,s^{tk}$ are all distinct, it follows that
$Q=aX_0^{t-i}X_2^i$ for some $a\in \bC^*$ and some
$i\in\{0,\dots,t\}$.  This is impossible because $P$ is
not divisible by $X_0$ nor by $X_2$.
\end{proof}



\section{Construction of $\bQ$-subvarieties of dimension 0}
 \label{sec:cons}

Throughout this section, we assume that the hypotheses of Theorem \ref{intro:thm}
hold for the current choice of $(\xi,\eta)\in\bC\times\bC^*$ and
$(r,s)\in\bQ\times\bQ^*$ with $r\neq 0$ and $|s|>1$, and for some
choice of parameters $\beta,\sigma,\nu\in\bR$ satisfying the
conditions \eqref{intro:thm:eq1} of the theorem.
We also fix a choice of polynomials $\tP_D$ as in Proposition
\ref{not:prop}, say one for each integer $D\ge D_0$, for some
fixed $D_0\in\bN^*$.  For those $D$, we define
\[
 W_D = \cZ(\Phi^i(\tP_D); \ 0 \le i < 2\lfloor D\rfloor^\sigma).
\]
We first establish the following analog of \cite[Prop.~6.4]{R2012}.

\begin{prop}
 \label{cons:prop1}
Let $D\in\bN^*$ with $D\ge D_0$, and let
$T = \lfloor D^{\sigma} \rfloor$.  If $W_D$ is not empty,
then any $\bQ$-subvariety $Z$ of\/ $\bP^2(\bC)$ contained in $W_D$
has dimension $0$ and, if $D$ is large enough, it satisfies
\[
 \deg(\tau^i(Z)) \le 2D^{2-\sigma}
 \et
 h(\tau^i(Z)) \le 6D^{1+\beta-\sigma}
 \quad
 (|i|<3T).
\]
\end{prop}

\begin{proof}
The argument is similar to the proof of
\cite[Prop.~6.4]{R2012} and simpler.  Put $P=\tP_D$ and
assume that $W_D\neq\emptyset$.  By Lemma \ref{height:lemma:gcd},
the polynomials $P,\Phi(P),\dots,\Phi^D(P)$ are relatively
prime.  So, as they are homogeneous of degree $D$, there
exist integers $a_1,\dots,a_D$ of absolute values at most
$D$ such that $Q:=\sum_{i=1}^D a_i\Phi^i(P)$ is relatively
prime to $P$.  Then, $W:=\cZ(P,Q)$ has dimension $0$.
So, it is a finite union of $\bQ$-subvarieties of $\bP^2(\bC)$
whose sum of the degrees is $\deg(W)\le D^2$ and whose sum
of the heights is
\[
 h(W)\le D\log\|P\|+D\log\|Q\|+\cO(D^2)
\]
(since the product of their Chow forms in degree $1$ divides
the polynomial map $G\colon\bC[\uX]_1\to\bC$ given by
$G(L)=\Res_{(D,D,1)}(P,Q,L)$ where
$\Res_{(D,D,1)}$ denotes the resultant in degrees $(D,D,1)$).
Assuming that $D$ is large enough, this gives $h(W)\le
5D^{1+\beta}$ thanks to Proposition \ref{not:prop}.

Let $Z$ be a $\bQ$-subvariety of\/ $\bP^2(\bC)$ contained
in $W_D$.  Since $D\le T$, we have $\tau^i(Z)\subseteq W$
for $i=0,\dots,T-1$.  Thus $Z$ has dimension $0$.  If the
sets $\tau^i(Z)$ ($i\in\bZ$) are not all distinct, then
there exists a positive integer $k$ such that $\tau^k(Z)=Z$.
As $Z$ is a finite set, we may further choose $k$ so that $\tau^k$
fixes each element of $Z$.  Then $Z$ consists of a single
point $(0:1:0)$ or $(0:0:1)$, since the latter are the only
points of \/$\bP^2(\bC)$ fixed by a power $\tau^k$ of $\tau$
with $k\in\bN^*$, and since their coordinates are rational.
In that case, we conclude that $\deg(Z)=1$ and $h(Z)=0$,
and the proposition is verified.  Thus, we may assume that
$Z,\tau(Z),\dots,\tau^{T-1}(Z)$ are distinct subvarieties
of \/$\bP^2(\bC)$.  As they are contained in $W$, we conclude that
\[
 \sum_{i=0}^{T-1} \deg(\tau^i(Z)) \le \deg(W)
 \et
 \sum_{i=0}^{T-1} h(\tau^i(Z)) \le h(W).
\]
By Lemma \ref{height:lemma1}, this implies that, for any
integer $i$ with $|i|< 3T$, we have
\[
 \deg(\tau^i(Z))
   = \deg(Z)
   \le \frac{\deg(W)}{T} = \frac{D^2}{T}
\]
and
\[
 h(\tau^i(Z))
  = h(Z)+\cO(T\deg(Z))
  \le \frac{h(W)}{T}+\cO(T\deg(Z))
  \le \frac{5D^{1+\beta}}{T}+\cO(D^2).
\]
The conclusion follows since $\beta>1+\sigma$.
\end{proof}

For each integer $D$ with $D\ge D_0$, we define
\[
 \cC_D
  = \left\{ P \in \bC[\uX]_D \,;\,
      \|P\| \le e^{2D^{\beta}},\,
       \max_{0 \le i < \lfloor D^{\sigma} \rfloor}
           |P(\ugamma_i)| \le e^{-(1/2)D^{\nu}}
    \right\}.
\]
This is a \emph{convex body} of $\bC[\uX]_D$, namely
a compact subset of the vector space $\bC[\uX]_D$
with non-empty interior, which satisfies $\lambda P + \mu Q
\in \cC_D$ for any $P,Q\in \cC_D$
and any $\lambda,\mu\in\bC$ with $|\lambda| + |\mu| \le 1$.
For any $t\in\{0,1,2\}$, and any $\bQ$-subvariety $Z$ of \/$\bP^2(\bC)$
of dimension $t$, we recall from \cite{LR} that the height
of $Z$ relative to $\cC_D$ is defined by
\[
 h_{\cC_D}(Z)
  = h_{\cC_D}(F)
  = \log \sup\{ |F(P_0,\dots,P_t)| \,;\, P_0,\dots,P_t\in \cC_D \}
\]
where $F$ denotes the Chow form of $Z$ in degree $D$.

\begin{prop}
\label{cons:prop2}
For each sufficiently large integer $D$, there exists a $\bQ$-subvariety
$Z_D$ of\/ $\bP^2(\bC)$ of dimension $0$ contained in $W_D$ with
\[
 h_{\cC_D}(Z_D)
 \le - \frac{1}{30} D^{\nu-\beta+\sigma-2} (2D^{\beta}\deg(Z_D) + D h(Z_D)).
\]
\end{prop}

\begin{proof}
The argument follows very closely the proofs of Propositions 6.1
and 6.2 of \cite{R2012} and so we will simply explain how
these need to be modified in order to yield the present
statement.

Set $T=\lfloor D^{\sigma} \rfloor$.  We first apply Theorem~5.6
of \cite{R2012} to the ideal $I^{(T)}$ for the subsets of the
group $\bC\times\bC^*$ given by
\[
 \Sigma = \{ (\xi + i r, \eta s^i)\,;\,  0 \le i <T \}
 \et
 \Sigma_1 = \{ (i r, s^i) \,;\, \ 0 \le i \le D \}.
\]
The projections of $\Sigma_1$ on each factor of $\bC\times\bC^*$
have cardinality $D+1>D$ while the sumset $\Sigma+\Sigma_1$ has
cardinality $D+T\le D+D^\sigma<\binom{D+2}{2}$, assuming $D$ large enough.
Since the polynomials of $I^{(T)}$ vanish at each point $(1,\gamma)$
with $\gamma\in\Sigma$, we conclude that the resultant
in degree $D$ viewed as a polynomial map $\Res_D\colon
\bC[\uX]_D^3\to \bC $ vanishes up to order $T$ at each
triple of elements of $ I_D^{(T)}$.

We then argue as in the proof of \cite[Prop.~6.1]{R2012}
using $Y=2D^\beta$ and $U=(1/2)D^\nu$ and replacing
everywhere the differential operator $\cD$ with the
translation morphism $\Phi$.  We also use our interpolation
result, Proposition \ref{interp:prop1}, in replacement of
\cite[Prop.~3.3]{R2012}.  Then, assuming that $D$ is
sufficiently large, all estimates work out and we obtain
\[
 h_{\cC_D}(\bP^2(\bC)) \le  -TU + 3YD^2+ 21 \log(3) D^3.
\]
From there, we follow almost word for word the proof of
\cite[Prop.~6.2]{R2012} for the choice of $P=\tP_D$.
We simply need to note that, by Proposition
\ref{not:prop}, the polynomials
$P,\Phi(P),\dots,\Phi^D(P)$ all belong
$\cC_D\cap\bZ[\uX]_D$ and that, by Lemma \ref{height:lemma:gcd},
they have no common irreducible factor in $\bQ[\uX]$.
Since $\beta>1$ and $\nu+\sigma > 2+\beta$, this implies,
for $D$ sufficiently large, the existence of a $\bQ$-subvariety
$Z_D$ of \/$\bP^2(\bC)$ of dimension $0$ contained in
$W_D=\cZ(\Phi^i(P); \ 0 \le i < 2T)$ with
\[
 h_{\cC_D}(Z_D)
 \le - \frac{TU}{7\,D^2Y}(Y\deg(Z_D) + D h(Z_D))
 \le - \frac{1}{30} D^{\nu-\beta+\sigma-2}
          (2D^{\beta}\deg(Z_D) + D h(Z_D)).
 \qedhere
\]
\end{proof}

At the expense of replacing $D_0$ by a larger integer
if necessary, we may assume that the above proposition
applies to each $D\ge D_0$.  For each such integer $D$,
we fix a corresponding choice of $Z_D$.


\begin{cor}
 \label{cons:cor}
If $D\ge D_0$ is sufficiently large, then, upon writing
$T=\lfloor D^{\sigma} \rfloor$, we have
\begin{equation*}
 \sum_{\alpha \in Z_D}
   \Big( D^\beta+ \log \, \dist(\alpha,\{\gamma_0,\dots,\gamma_{T-1}\})\Big) 			
 \le - \frac{1}{30} D^{\nu -\beta + \sigma -2}
      \big(2D^{\beta} \deg(Z_D) + Dh(Z_D)\big).
\end{equation*}
\end{cor}

\begin{proof}
Since $Z_D$ is a $\bQ$-subvariety of \/$\bP^2(\bC)$ of dimension $0$,
Proposition 2.3 of \cite{R2012} gives
\[
 \sum_{\ualpha\in\uZ_D} \log \sup\{|P(\ualpha)|\,;\, P\in\cC_D\}
 \le h_{\cC_D}(Z_D)-Dh(Z_D)+9\log(3)D\deg(Z_D)
\]
where $\uZ_D$ denote a set of representatives of the points of
$Z_D$ by elements of $\bC^3$ of norm $1$.  For each
$P\in I_D^{(T)}$ with $\|P\|=1$, we have $e^{2D^\beta}P\in\cC_D$.
So, for any $\alpha\in Z_D$ with representative $\ualpha\in\uZ_D$,
we obtain
\[
 2D^\beta + \log |I_D^{(T)}|_\alpha
 \le \log \sup\{|P(\ualpha)|\,;\, P\in\cC_D\}.
\]
Moreover, let $\cS=\{\gamma_0,\dots,\gamma_{T-1}\}$.
If $D$ is large enough, we have $T\le D^\sigma\le \binom{D+1}{2}$,
and Proposition \ref{interp:prop:dist} gives
\[
 \log \dist(\alpha,\cS)
 \le T^{3/2}\log(c_3)+\log |I_D^{(T)}|_\alpha
 \le \frac{D^\beta}{2}+\log |I_D^{(T)}|_\alpha
\]
using $(3/2)\sigma<1+\sigma<\beta$.  Combining these estimates,
we obtain
\begin{align*}
 \sum_{\alpha \in Z_D}
   \Big( D^\beta+ \log \, \dist(\alpha,\cS)\Big)
& \le
   -\frac{D^\beta}{2}\deg(Z_D)
   + \sum_{\ualpha \in \uZ_D}
       \log \sup\{|P(\ualpha)|\,;\, P\in\cC_D\} \\
& \le h_{\cC_D}(Z_D),
\end{align*}
and the conclusion follows using the upper bound for $h_{\cC_D}(Z_D)$
provided by Proposition \ref{cons:prop2}.
\end{proof}

For each integer $T\ge 1$, we define
\[
 \Gamma(T)=\{\gamma_i\,;\, |i|<T\}
=\{\gamma_{-(T-1)},\dots,\gamma_{T-1}\}.
\]
We now come to the main result of this section.

\begin{prop}
 \label{cons:prop:m}
If $D\ge D_0$ is sufficiently large, there exists an integer
$m$ with $0\le m < \lfloor D^\sigma\rfloor$ for which the
translate $\tZ_D=\tau^{-m}(Z_D)$ satisfies
\[
 \sum_{\alpha\in\tZ_D} \log\dist(\alpha,\Gamma(T^*))
 \le -\frac{T^*}{120}D^{\nu -\beta -2}
       \big(D^{\beta}\deg(\tZ_D)+Dh(\tZ_D)\big),
\]
for each integer $T^*$ with $1\le T^*\le D^\sigma$.
\end{prop}

\begin{proof}
Fix $D$ large enough so that we have both
$D^{\beta-\sigma-1} \ge \max\{c_1,2c_4\}$ and $D^{\nu-2}\ge 240$,
and that Corollary \ref{cons:cor} applies.
Let $T=\lfloor D^\sigma\rfloor$.  For each $\alpha\in Z_D$,
choose an index $t(\alpha)$ with $0\le t(\alpha)<T$ such that
$\dist(\alpha,\gamma_{t(\alpha)})$ is minimal and set
\[
 \delta(\alpha)
  = \min\big\{0,\ D^\beta+\log \dist(\alpha,\gamma_{t(\alpha)})\big\}.
\]
With this notation, Corollary \ref{cons:cor} yields
\begin{equation*}
 \sum_{\alpha\in Z_D} \delta(\alpha) \le -4D^\sigma B
 \quad
 \text{where}
 \quad
 B = \frac{1}{120}D^{\nu-\beta-2}\big(2D^\beta\deg(Z_D)+Dh(Z_D)\big).
\end{equation*}
Let $k$ be the smallest integer with $2^k\ge D^\sigma$ and set
$I_k=\{0,\dots,2^k-1\}$.  By the above, we have
\[
 \sum_{\{\alpha\in Z_D\,;\,t(\alpha)\in I_k\}} \delta(\alpha)
 \le -2^{k+1} B.
\]
Starting from $I_k$, we choose recursively a descending sequence
of sets $I_k\supseteq I_{k-1}\supseteq\cdots\supseteq I_0$ such that,
for each $j=0,\dots,k$, the set $I_j$ consists of $2^j$ consecutive
integers and
\[
 \sum_{\{\alpha\in Z_D\,;\,t(\alpha)\in I_j\}} \delta(\alpha)
 \le -2^{j+1} B.
\]
Once $I_j$ has been constructed for some index $j$ with
$1\le j\le k$, it suffices to take for $I_{j-1}$ the set consisting
of the $2^{j-1}$ smallest elements of $I_j$ or its complement depending
on which yields the smallest sum.  In particular, $I_0$ consists of a
single integer $m$.  For $j=0,\dots,k$, this integer belongs to $I_j$.
So, we obtain $I_j \subseteq \{t\in\bZ\,;\, |t-m|<2^j\}$ and thus
\begin{equation}
 \label{cons:prop:m:eq3}
 \sum_{\{\alpha\in Z_D\,;\,|t(\alpha)-m|<2^j\}} \delta(\alpha)
 \le -2^{j+1} B
\end{equation}
because $\delta(\alpha)\le 0$ for each $\alpha\in Z_D$.  We also note
that $0\le m<T$ since applying \eqref{cons:prop:m:eq3} with $j=0$
shows that $m=t(\alpha)$ for at least one $\alpha\in Z_D$.

We claim that the translate $\tZ_D=\tau^{-m}Z_D$ has the right
property.  To show this, fix an integer $T^*$ with
$1\le T^*\le T$, and set
\[
 \cA=\{\alpha\in Z_D\,;\ |t(\alpha)-m|<T^*\}.
\]
Since the projective distance between any two points is at most $2$,
we find
\[
 \sum_{\alpha\in \tZ_D} \log\frac{\dist(\alpha,\Gamma(T^*))}{2}
  = \sum_{\alpha\in Z_D}
         \log\frac{\dist(\tau^{-m}(\alpha),\Gamma(T^*))}{2}
  \le \sum_{\alpha\in \cA}
         \log\frac{\dist(\tau^{-m}(\alpha),\Gamma(T^*))}{2}.
\]
For any $\alpha\in\cA$, Lemma \ref{not:lemma2} gives
\[
 \log\dist(\tau^{-m}(\alpha),\Gamma(T^*))
  \le c_1 m + \log\dist\big(\alpha,\tau^m(\Gamma(T^*))\big).
\]
Since $c_1 m< c_1D^\sigma \le D^\beta$ and
$\gamma_{t(\alpha)}\in \tau^m(\Gamma(T^*))$, this implies that
\[
 \log\frac{\dist(\tau^{-m}(\alpha),\Gamma(T^*))}{2}
   \le \min\big\{0,\
        D^\beta + \log\dist\big(\alpha,\gamma_{t(\alpha)}\big)
        \big\}
   = \delta(\alpha).
\]
Thus we conclude that
\begin{equation}
 \label{cons:prop:m:eq4}
 \sum_{\alpha\in \tZ_D} \log \dist(\alpha,\Gamma(T^*))
  \le \deg(\tZ_D) + \sum_{\alpha\in \cA} \delta(\alpha).
\end{equation}
To complete the proof, we note that $2^j\le T^*\le 2^{j+1}$
for some integer $j$ with $0\le j< k$.  Then, the set $\cA$
contains all $\alpha\in Z_D$ with $|t(\alpha)-m|<2^j$ and,
using \eqref{cons:prop:m:eq3}, we obtain
\begin{equation}
 \label{cons:prop:m:eq5}
 \sum_{\alpha\in \cA} \delta(\alpha)
  \le \sum_{\{\alpha\in Z_D\,;\,|t(\alpha)-m|<2^j\}} \delta(\alpha)
  \le -2^{j+1} B
  \le -T^*B.
\end{equation}
By Lemma \ref{height:lemma1}, we have $\deg(Z_D)=\deg(\tZ_D)$ and
$h(Z_D)\ge h(\tZ_D)-c_4m\deg(Z_D)$.  Since $m \le D^\sigma$,
this gives $Dh(Z_D)\ge Dh(\tZ_D)-(1/2)D^\beta\deg(\tZ_D)$ and so
\[
 B \ge \frac{1}{120} D^{\nu - \beta -2}
       \big(D^\beta\deg(\tZ_D)+Dh(\tZ_D)\big)
       + \frac{D^{\nu-2}}{240}\deg(\tZ_D).
\]
The requested estimate follows by combining \eqref{cons:prop:m:eq4}
and \eqref{cons:prop:m:eq5} with this inequality, using the
hypothesis that $D^{\nu-2}\ge 240$.
\end{proof}

Again, we may adjust $D_0$ so that Propositions \ref{cons:prop1}
and \ref{cons:prop:m} apply to each integer $D$ with $D\ge D_0$.
For each of those integers $D$, we fix a choice of translate
$\tZ_D$ of $Z_D$, as in Proposition \ref{cons:prop:m}.  Since
$h(\tZ_D)\ge 0$, we readily deduce the following consequence.

\begin{cor}
 \label{cons:cor2}
For any integers $D,T^*$ with $D\ge D_0$ and $0\le T^*\le \lfloor
D^\sigma\rfloor$, there exists a point $\alpha$ of $\tZ_D$
such that
\[
 \log\dist(\alpha,\Gamma(T^*))
   \le -\frac{T^*D^{\nu-2}}{120}.
\]
\end{cor}

For the choice of $T^*=1$, this reduces to the following statement.

\begin{cor}
 \label{cons:cor3}
For each integer $D$ with $D\ge D_0$, there exists $\alpha\in\tZ_D$ such that
\[
 \log \dist(\alpha,\gamma_0)
   \le -\frac{D^{\nu-2}}{120}.
\]
\end{cor}


\section{Proof of the main theorem}
\label{sec:proof}

As in the previous section, we assume that the hypotheses of
Theorem \ref{intro:thm} are satisfied and that $|s|>1$.  If
$\xi,\eta\in\Qbar$, then Liouville's inequality implies that
$P_D(\xi+ir,\eta s^i)=0$ for each $i=0,\dots,\lfloor D^{\beta-1}\rfloor$,
when $D$ is large enough.  So, from now on, we
further assume that $(\xi, \eta) \notin \Qbar\times\Qbar^*$,
and seek for a contradiction.  The latter hypothesis implies that
$\dist(\alpha,\gamma_0)>0$ for each $\alpha\in\bP^2(\Qbar)$
and so, by virtue of Corollary \ref{cons:cor3}, we deduce that
\[
 \lim_{D\to\infty} \max\{\deg(\tZ_D),\, h(\tZ_D)\}
 = \infty.
\]

Let $D\ge D_0$ be an arbitrarily large integer, and let $D^*\ge 0$
be any integer satisfying
\begin{equation}
 \label{proof:eq1}
 \deg(\tZ_D) > 2(D^*)^{2-\sigma}
 \quad\text{or}\quad
 h(\tZ_D) > 6 (D^*)^{1+\beta-\sigma}.
\end{equation}
If $D$ is sufficiently large, we may choose $D^*\ge D_0$
and Proposition \ref{cons:prop1} gives
$\tau^i(\tZ_D)\nsubseteq W_{D^*}$ for each integer $i$ with
$|i|<3T^*$, where $T^*=\lfloor(D^*)^\sigma\rfloor$.
Since $\tau^m(\tZ_{D^*})=Z_{D^*} \subseteq W_{D^*}$ for some
$m\in\bZ$ with $0\le m < T^*$, this implies that $\tau^i(\tZ_D)
\neq \tau^m(\tZ_{D^*})$ for each $i$ with $|i|<3T^*$, and thus
\begin{equation}
 \label{proof:eq2}
 \tau^i(\tZ_D) \neq \tau^j(\tZ_{D^*})
 \quad
 (|i|,|j|<T^*).
\end{equation}
In particular, this means that $D^*\neq D$.  Thus,
if $D$ is large enough, the largest integer $D^*$ satisfying
\eqref{proof:eq1} lies in the range $D_0\le D^*< D$ and,
for that choice of $D^*$, we also have
\begin{equation}
 \label{proof:eq3}
 \deg(\tZ_D) \le 3(D^*)^{2-\sigma}
 \et
 h(\tZ_D) \le 7 (D^*)^{1+\beta-\sigma}.
\end{equation}
If $\sigma\ge 3/2$, this yields
$\deg(\tZ_D) \le 3(D^*)^{2-\sigma}\le 3(D^*)^{\sigma-1}$.
This is therefore a situation where the following
hypothesis is fulfilled.

\subsection*{Case 1.}
Suppose that, for each sufficiently large integer $D$,
we have $\deg(\tZ_D)\le 3(D^*)^{\sigma-1}$ with the
above choice of $D^*$.

As $D^*$ tends to infinity with $D$, this implies that
we also have $\deg(\tZ_{D^*})\le 3(D^*)^{\sigma-1}$ if $D$
is large enough.  By Corollary \ref{cons:cor2}, there
exist points $\alpha\in\tZ_D$ and $\alpha^*\in\tZ_{D^*}$
such that
\[
 \log\dist(\alpha,\Gamma(T^*))
   \le -\frac{T^*D^{\nu-2}}{120}
 \et
 \log\dist(\alpha^*,\Gamma(T^*))
   \le -\frac{T^*(D^*)^{\nu-2}}{120},
\]
where $T^*=\lfloor (D^*)^\sigma\rfloor$.
Choose integers $i,j$ with $|i|,|j|<T^*$ such that
$\gamma_i$ and $\gamma_j$ are respectively closest
to $\alpha$ and $\alpha^*$ within the set $\Gamma(T^*)$.
Using Lemma \ref{not:lemma2} and assuming that $D^*$
is large enough, we find
\[
 \log \dist(\tau^{-i}(\alpha),\gamma_0)
 \le c_1T^* + \log\dist(\alpha,\Gamma(T^*))
 \le -\log(3) - \frac{1}{150}(D^*)^{\nu+\sigma-2}
\]
and the same with $i$ and $\alpha$ replaced respectively
by $j$ and $\alpha^*$.  Thanks to the weak triangle inequality
satisfied by the distance, this yields
\begin{equation}
\label{proof:case1:eq1}
 \log \dist(\tau^{-i}(\alpha),\tau^{-j}(\alpha^*))
 \le -\frac{1}{150}(D^*)^{\nu+\sigma-2}.
\end{equation}
Consider the $\bQ$-subvarieties of \/$\bP^2(\bC)$ given by
$Z=\tau^{-i}(\tZ_D)$ and $Z^*=\tau^{-j}(\tZ_{D^*})$.
They contain respectively the points $\tau^{-i}(\alpha)$
and $\tau^{-j}(\alpha^*)$ and, by \eqref{proof:eq2}, they are
distinct.  Moreover, their degrees are the same as those
of $\tZ_D$ and $\tZ_{D^*}$ respectively.  So they are bounded
above by $3(D^*)^{\sigma-1}$.  Assuming $D^*$ large enough, we
also find
\[
 h(Z) \le c_4T^*\deg(\tZ_D) + h(\tZ_D)
       \le 8(D^*)^{1+\beta-\sigma}
\]
and similarly $h(Z^*) \le 8(D^*)^{1+\beta-\sigma}$, using
Lemma \ref{height:lemma1} together with \eqref{proof:eq3}.
Applying Proposition \ref{height:prop} to the singleton
$\cA=\{(\tau^{-i}(\alpha),\tau^{-j}(\alpha^*))\}$, we thus obtain
\[
 \log \dist(\tau^{-i}(\alpha),\tau^{-j}(\alpha^*))
 \ge -63(D^*)^{2\sigma-2}-48(D^*)^{\beta}.
\]
Since $\nu>2+\beta-\sigma$ and $\beta>\sigma>2\sigma-2$, this
contradicts \eqref{proof:case1:eq1} if $D^*$ is sufficiently
large or, equivalently, if $D$ is sufficiently large.

\subsection*{Case 2.} As the previous case is ruled out,
there exist arbitrarily large values of $D$ for which
$\deg(\tZ_D)> 3(D^*)^{\sigma-1}$.

In view of the observation preceding Case 1, this implies that
$\sigma<3/2$.

Put $Z=\tau^{T^*-1}(\tZ_D)$.  As observed at the beginning of the
proof, this $\bQ$-subvariety of \/$\bP^2(\bC)$ is not contained in $W_{D^*}$
and so there exists an integer $j$ with $0\le j<2T^*$ such that
the polynomial $P=\Phi^j(\tP_{D^*})\in \bZ[\uX]_{D^*}$ does not
belong to the ideal of $Z$.  By Proposition 2.3 of \cite{R2012},
this implies that
\[
 \sum_{\ualpha\in\uZ} \log |P(\ualpha)|
 \ge -7\log(3)D^*\deg(Z)-D^*h(Z)
\]
where $\uZ$ denotes a set of representatives of the elements of
$Z$ by points of $\bC^3$ of norm $1$.  By Lemma
\ref{height:lemma1}, we also have $\deg(Z)=\deg(\tZ_D)$ and
$h(Z)\le c_4T^*\deg(\tZ_D)+h(\tZ_D)$.  Therefore, assuming $D^*$
sufficiently large, the previous estimate gives
\begin{equation}
 \label{proof:case2:eq1}
 \sum_{\ualpha\in\uZ} \log |P(\ualpha)|
 \ge -(c_4+7\log 3)(D^*)^{\sigma+1}\deg(\tZ_D) - D^*h(\tZ_D).
\end{equation}

Fix temporarily a point $\alpha\in Z$, set $\talpha =
\tau^{1-T^*}(\alpha) \in \tZ_D$, choose an integer
$i$ with $|i| < T^*$ for which $\dist(\talpha,\gamma_i)$
is minimal, and set $k=i+T^*-1$.  By Lemma \ref{not:lemma3},
we have
\[
 |P(\ualpha)|
  \le \left|P\left(\frac{\ugamma_k}{\|\ugamma_k\|}\right)\right|
       +D^*\binom{D^*+2}{2}\|P\|\dist(\alpha,\gamma_k)
\]
where $\ualpha$ denotes the representative of $\alpha$ in $\uZ$.
Since $0\le j,k<2T^*$, Proposition \ref{not:prop} gives $\|P\|\le
\exp(2(D^*)^\beta)$ and
\[
 |P(\ugamma_k)|
   = |\tP_{D^*}(\ugamma_{j+k})|
   \le \exp(-(1/2)(D^*)^\nu).
\]
By Lemma \ref{not:lemma2}, we also have
\[
 \dist(\alpha,\gamma_k)
  \le c_1T^*+\dist(\talpha,\gamma_i)
   = c_1T^*+\dist(\talpha,\Gamma(T^*)).
\]
Combining these estimates, we conclude that
\begin{equation}
\label{proof:case2:eq2}
 \log |P(\ualpha)|
  \le \max\left\{-\frac{(D^*)^\nu}{3},\,
       3(D^*)^\beta + \log \frac{\dist(\talpha,\Gamma(T^*))}{2}\right\}
\end{equation}
if $D^*$ is sufficiently large (independently of the choice of $\alpha$).

Suppose first that there exists a point $\ualpha_0\in\uZ$ for which
$\log |P(\ualpha_0)| \le -(D^*)^\nu/3$.  Then, by the above, we obtain
\[
 \sum_{\ualpha\in\uZ} \log |P(\ualpha)|
 \le -\frac{1}{3}(D^*)^\nu + 3(D^*)^\beta\deg(Z),
\]
since the distance is bounded above by $2$.
If $D^*$ is sufficiently large, this contradicts \eqref{proof:case2:eq1}
in view of the estimates \eqref{proof:eq3} and of the fact that
$\nu>2+\beta-\sigma$.

We may therefore assume that $\log |P(\ualpha)| > -(D^*)^\nu/3$ for each
$\ualpha\in\uZ$ and so \eqref{proof:case2:eq2} yields
\begin{align*}
 \sum_{\ualpha\in\uZ} \log |P(\ualpha)|
 &\le 3(D^*)^\beta\deg(Z)
      + \sum_{\talpha\in\tZ_D}\log\dist(\talpha,\Gamma(T^*))\\
 &\le 3(D^*)^\beta\deg(\tZ_D)
     - \frac{T^*}{120}D^{\nu-\beta-2} \big(D^{\beta}\deg(\tZ_D)+Dh(\tZ_D)\big)
\end{align*}
where the second estimate follows from Proposition \ref{cons:prop:m}
together with the fact that $Z$ and $\tZ_D$ have the same degree.
Comparing this upper bound with \eqref{proof:case2:eq1} and assuming
that $D$ is sufficiently large, we deduce that
\[
 D^*h(\tZ_D)
   \ge  \frac{(D^*)^\sigma}{150}
       D^{\nu-\beta-2} \big(D^{\beta}\deg(\tZ_D)+Dh(\tZ_D)\big),
\]
thus
\[
 D^*h(\tZ_D)
   \ge \frac{(D^*)^\sigma}{150} D^{\nu-2} \deg(\tZ_D)
 \et
 D^*
   \ge \frac{(D^*)^\sigma}{150} D^{\nu-\beta-1}.
\]
Using the upper bound for $h(\tZ_D)$ provided by \eqref{proof:eq3}
and the hypothesis that $\deg(\tZ_D)\ge 3(D^*)^{\sigma-1}$, these
inequalities yield
\[
 350 (D^*)^{3+\beta-3\sigma}\ge D^{\nu-2}
 \et
 150 D^{1+\beta-\nu}\ge (D^*)^{\sigma-1}.
\]
As this holds for arbitrarily large values of $D$, we conclude that
$1+\beta-\nu\ge 0$ and
\[
 (3+\beta-3\sigma)(1+\beta-\nu) \ge (\sigma-1)(\nu-2).
\]
Upon writing $\nu=2+\beta-\sigma+\delta$, this inequality is equivalent
to $(2+\beta-2\sigma)\delta\le (\sigma-1)(3-2\sigma)$ which contradicts
the last of the conditions \eqref{intro:thm:eq1} in Theorem \ref{intro:thm},
since $\sigma<3/2$.  This final contradiction completes the proof
of the theorem.


\end{document}